\documentclass{amsart}

\usepackage{ adjustbox }
\usepackage{ amsthm }
\usepackage{ amssymb }
\usepackage{ amsmath }
\usepackage{ array }
\usepackage{ booktabs }
\usepackage{ cite }
\usepackage{ color }
\usepackage{ enumitem }
\usepackage{ latexsym }
\usepackage{ url }
\usepackage[all]{ xy }

\theoremstyle{definition}

\newtheorem{definition}{Definition} 

\theoremstyle{plain}

\newtheorem{lemma}[definition]{Lemma}

\newtheorem{theorem}[definition]{Theorem}
\newtheorem{corollary}[definition]{Corollary}

\allowdisplaybreaks

\begin{document}

\title{On tortkara triple systems}

\author{Murray Bremner}

\address{Department of Mathematics and Statistics, University of Saskatchewan, Canada}

\email{bremner@math.usask.ca}




\thanks{The research of the author was supported by a Discovery Grant from NSERC, 
the Natural Sciences and Engineering Research Council of Canada.}

\begin{abstract}
The commutator $[a,b] = ab - ba$ in a free Zinbiel algebra (dual Leibniz algebra) is 
an anticommutative operation which satisfies no new relations in arity 3.
Dzhumadildaev discovered a relation $T(a,b,c,d)$ which he called the tortkara identity 
and showed that it implies every relation satisfied by the Zinbiel commutator in arity 4.
Kolesnikov constructed examples of anticommutative algebras satisfying $T(a,b,c,d)$ 
which cannot be embedded into the commutator algebra of a Zinbiel algebra.
We consider the tortkara triple product $[a,b,c] = [[a,b],c]$ in a free Zinbiel algebra 
and use computer algebra to construct a relation $TT(a,b,c,d,e)$ which implies every
relation satisfied by $[a,b,c]$ in arity 5.
Thus, although tortkara algebras are defined by a cubic binary operad (with no Koszul
dual), the corresponding triple systems are defined by a quadratic ternary operad (with 
a Koszul dual).
We use computer algebra to construct a relation in arity 7 satisfied by $[a,b,c]$ which 
does not follow from the relations of lower arity.
It remains an open problem to determine whether there are further new identities in arity $\ge 9$.
\end{abstract}

\maketitle


\section{Introduction}

Leibniz algebras were first studied by Blokh \cite{Blokh1965} under the name D-algebras
and were given their present name by Loday \cite{Loday1993}.
A vector space $L$ over a field $\mathbb{F}$ with a bilinear multiplication $L \times L \to L$
denoted $(a,b) \mapsto [a,b]$ is called a (left) Leibniz algebra if it satisfies the (left) 
derivation identity, 
\begin{equation}
[a,[b,c]] \equiv [[a,b],c] + [b,[a,c]],
\end{equation}
where the symbol $\equiv$ indicates that the equation holds for all values of the arguments.
The corresponding symmetric operad is quadratic and hence possesses a Koszul dual which was 
first studied by Loday \cite{Loday1995}.
Algebras over the Zinbiel (dual Leibniz operad) are defined by the (right) Zinbiel identity,
\begin{equation}
(ab)c \equiv a(bc) + a(cb).
\end{equation}
This identity first appeared in work of Sch\"utzenberger \cite[equation (S2), p.~18]{Schutzenberger1959}
on combinatorics of free Lie algebras.
In every Zinbiel algebra the anticommutator $a \circ b = ab + ba$ is associative (and commutative):
\begin{align*}
&
( a \circ b ) \circ c - a \circ ( b \circ c )
\\[-1pt]
&=
(ab)c + (ba)c + c(ab) + c(ba) - a(bc) - a(cb) - (bc)a - (cb)a
\\[-1pt]
&=
a(bc) + a(cb) + b(ac) + b(ca) + \cdots - b(ca) - b(ac) - c(ba) - c(ab)
= 0.
\end{align*}
Zinbiel algebras have therefore been called precommutative algebras; 
see Aguiar \cite{Aguiar2000}.
Zinbiel algebras are commutative dendriform (preassociative) algebras.
Kawski \cite{Kawski2009} discusses applications of Zinbiel algebras in control theory.
Livernet \cite{Livernet1998} studies a generalization of rational homotopy based on 
Leibniz and Zinbiel algebras.

An important feature shared by both the Leibniz and Zinbiel operads is that they are regular
in the sense that in every arity $n$ the homogeneous component is isomorphic to the 
regular $S_n$-module $\mathbb{F} S_n$.
Indeed, both the Leibniz and Zinbiel identities can be interpreted as directed rewrite rules,
\begin{equation}
[[a,b],c] \longmapsto [a,[b,c]] - [b,[a,c]],
\qquad\quad
(ab)c \longmapsto a(bc) + a(cb).
\end{equation}
These rules are the inductive step in the proof that every monomial in either operad can be 
rewritten as a linear combination of right-normed monomials (formed by using only left multiplications).
Thus in each arity we require only one association type, and so any (multilinear) monomial in
any association type can be identified with a linear combination of permutations of its arguments.
Regular parameterized one-relation operads have recently been classified by the author and 
Dotsenko \cite{BD2017}.

In both the Leibniz and the Zinbiel operads, the generating binary operation has no symmetry
(it is neither commutative nor anticommutative), so it is of interest to consider 
the same operads with a different set of generators, namely the commutator and anticommutator.
This process is called polarization and has been studied in detail by Markl \& Remm \cite{MR2006}.
Basic results on the polarizations of the Leibniz and Zinbiel operads are found in two papers by
Dzhumadildaev \cite{Askar2007,Askar2008}.
We make no further comment on the Leibniz case, since our focus in this paper is on the Zinbiel
operad, and in particular its binary and ternary suboperads generated by the Zinbiel commutator 
$[a,b] = ab-ba$ (the tortkara product) and the iterated Zinbiel commutator (the tortkara triple product):
\begin{equation}
\label{ttp}
\begin{array}{l@{\,}l}
[a,b,c] = [[a,b],c]
& 
= (ab)c - (ba)c - c(ab) + c(ba) 
\\[1mm]
&
= a(bc) + a(cb) - b(ac) - b(ca) - c(ab) + c(ba).
\end{array}
\end{equation}
The normal form of $[a,b,c]$ includes all permutations of $a, b, c$ in the association type
$\ast(\ast\ast)$; the sign pattern reflects the fact that both sides alternate in $a, b$.

Dzhumadildaev \cite{Askar2007} showed that the Zinbiel commutator does not satisfy any new 
(quadratic) relation in arity 3, but does satisfy new (cubic) relations of arity 4, all of which
are consequences of what is known as the tortkara identity:
\begin{equation}
\label{tortkaraidentity}
(ab)(cd) + (ad)(cb) \equiv J(a,b,c)d + J(a,d,c)b,
\end{equation}
where $J(a,b,c) = (ab)c + (bc)a + (ca)b$.
For further results on identical relations for Zinbiel algebras, and the speciality problem
for tortkara algebras, see Dzhumadildaev \& Tulenbaev \cite{DT2005}, Naurazbekova \& Umirbaev 
\cite{NU2010}, and Kolesnikov \cite{Kolesnikov2016}.
Since the suboperad of the Zinbiel operad generated by the commutator is cubic,
it is natural ask whether the operad of the corresponding triple systems is quadratic.

A similar case is that of Jordan algebras, where the original binary operad is 
defined by a commutative product $a \circ b$ satisfying the cubic Jordan identity,
which (if the characteristic of $\mathbb{F}$ is not 2) is equivalent to 
the multilinear identity,
\[
\begin{array}{l}
((a \circ b) \circ c) \circ d + ((b \circ d) \circ c) \circ a + ((d \circ a) \circ c) \circ b 
\equiv {}
\\[1pt]
(a \circ b) \circ (c \circ d) + (b \circ d) \circ (c \circ a) + (d \circ a) \circ (c \circ b).
\end{array}
\]
On the other hand, the Jordan triple product is defined by the following expression, which reduces 
to $abc + cba$ in any special Jordan algebra:
\[
\{a,b,c\} = \tfrac12 \big( ( a \circ b ) \circ c + a \circ ( b \circ c ) - b \circ ( a \circ c ) \big).
\]
This product satisfies a twisted form of the ternary derivation identity:
\[
\{a,b,\{c,d,e\}\} \equiv \{\{a,b,c\},d,e\} - \{c,\{b,a,d\},e\} + \{c,d,\{a,b,e\}\}.
\]
For further information on Jordan algebras and triple systems, 
we refer to McCrimmon \cite{McCrimmon2004} and Meyberg \cite{Meyberg1972};
see also Loos \& McCrimmon \cite{LM1977}.

In this paper we explain how computer algebra may be used to verify that 
the tortkara triple product \eqref{ttp} satisfies a (quadratic) relation of arity 5
which is not a consequence of the skewsymmetry in arity 3.
For the $S_5$-submodule of all such relations we determine an explicit generator  
(Theorem \ref{theorem5}) and the decomposition into irreducible representations
(Corollary \ref{corollary5}).
We then extend these computations to arity 7, and verify that there exist relations
in seven variables for the tortkara triple product which are new in the sense that 
they do not follow from the known relations of lower arity.
We determine an explicit new relation, which however does not generate all the new relations, 
and the decomposition into irreducible representations of the $S_7$-submodule of all new relations 
(Theorem \ref{theorem7}).
It is an open question whether there exist further new relations in arity $\ge 9$.

Although this paper is written from the point of view of algebraic operads, we require
very little background in that topic; the standard reference for the theoretical aspects
is Loday \& Vallette \cite{LV2012}, and for algorithmic methods the reader may refer to
the author \& Dotsenko \cite{BD2016}.


\section{Relations of arity 5}

We write $\mathbf{Zinb}$ for the symmetric operad generated by one binary operation denoted $ab$ 
with no symmetry satisfying the Zinbiel relation: 
\begin{equation}
\label{zinbiel}
(ab)c \equiv a(bc) + a(cb).
\end{equation}
We write $\mathbf{TTS}$ (tortkara triple system) for the suboperad of $\mathbf{Zinb}$ generated by 
the tortkara triple product \eqref{ttp}.
The operad $\mathbf{TTS}$  is a quotient of the free symmetric operad 
$\mathbf{SkewTS}$ generated by one ternary operation $[a,b,c]$ satisfying
\begin{equation}
[a,b,c] + [b,a,c] \equiv 0.
\end{equation}
We use the same symbol $[a,b,c]$ for the generators of both $\mathbf{TTS}$ and $\mathbf{SkewTS}$.

\begin{lemma}
\label{skewbasis5}
A basis for the homogeneous space $\mathbf{SkewTS}(5)$ of arity 5 consists of the following 
90 multilinear monomials where $\sigma \in S_5$ acts on $\{ a,b,c,d,e \}$:
\begin{itemize}[leftmargin=*]
\item
60 monomials $[[a^\sigma\!,b^\sigma\!,c^\sigma],d^\sigma\!,e^\sigma]$ where
$a^\sigma \prec b^\sigma$ in lex order, and 
\item
30 monomials $[a^\sigma\!,b^\sigma\!,[c^\sigma\!,d^\sigma\!,e^\sigma]]$ where
$a^\sigma \prec b^\sigma$ and $c^\sigma \prec d^\sigma$ in lex order.
\end{itemize}
\end{lemma}

\begin{proof}
Immediate.
\end{proof}

\begin{lemma}
\label{zinbielbasis5}
A basis for the homogeneous space $\mathbf{Zinb}(5)$ of arity 5 consists of the 120 multilinear
monomials $a^\sigma ( b^\sigma ( c^\sigma ( d^\sigma e^\sigma ) ) )$ where $\sigma \in S_5$ acts 
on $\{ a,b,c,d,e \}$; hence we may identify $\mathbf{Zinb}(5)$ with the regular 
$S_5$-module $\mathbb{F}S_5$.
\end{lemma}

\begin{proof}
Repeated application of the Zinbiel relation \eqref{zinbiel} as the rewrite rule
\[
(xy)z \,\longmapsto\, x(yz) + x(zy),
\]
allows us to express any nonassociative monomial as a linear combination of right-normed 
monomials.
There are 14 binary association types in arity 5:
\[
\begin{array}{lllll}
(((\ast\ast)\ast)\ast)\ast, &\quad
((\ast(\ast\ast))\ast)\ast, &\quad
((\ast\ast)(\ast\ast))\ast, &\quad
(\ast((\ast\ast)\ast))\ast, &\quad
(\ast(\ast(\ast\ast)))\ast,
\\
((\ast\ast)\ast)(\ast\ast), &\quad
(\ast(\ast\ast))(\ast\ast), &\quad
(\ast\ast)((\ast\ast)\ast), &\quad
(\ast\ast)(\ast(\ast\ast)),
\\
\ast(((\ast\ast)\ast)\ast), &\quad
\ast((\ast(\ast\ast))\ast), &\quad
\ast((\ast\ast)(\ast\ast)), &\quad
\ast(\ast((\ast\ast)\ast)), &\quad
\ast(\ast(\ast(\ast\ast))).
\end{array}
\]
We present explicit formulas for the normal forms of these association types 
with the identity permutation of the arguments $a, b, c, d, e$; 
in the first nine cases, the sum is over all permutations $\sigma \in S_4$ 
acting on $\{ a, b, c, d \}$;
in the next three cases, the sum is over all permutations $\tau \in S_3$ acting on $\{ c, d, e \}$;
the notation $x \prec y$ indicates that $x$ must appear to the left of $y$:
\begin{alignat*}{2}
&
(((ab)c)d)e = a \sum b^\sigma ( c^\sigma ( d^\sigma e^\sigma ) ),
&\qquad
&
((a(bc))d)e = a \sum_{b \prec c} b^\sigma ( c^\sigma ( d^\sigma e^\sigma ) ),
\\
&
((ab)(cd))e = a \sum_{c \prec d} b^\sigma ( c^\sigma ( d^\sigma e^\sigma ) ),
&\qquad
&
(a((bc)d))e = a \sum_{b \prec c, \, b \prec d} b^\sigma ( c^\sigma ( d^\sigma e^\sigma ) ), 
\\
&
(a(b(cd)))e = a \sum_{b \prec c \prec d} b^\sigma ( c^\sigma ( d^\sigma e^\sigma ) ),
&\qquad
&
((ab)c)(de) = a \sum_{d \prec e} b^\sigma ( c^\sigma ( d^\sigma e^\sigma ) ),
\\
&
(a(bc))(de) = a \sum_{b \prec c, \, d \prec e} b^\sigma ( c^\sigma ( d^\sigma e^\sigma ) ),
&\qquad
&
(ab)((cd)e) = a \sum_{c \prec d, \, c \prec e} b^\sigma ( c^\sigma ( d^\sigma e^\sigma ) ), 
\\
&
(ab)(c(de)) = a \sum_{c \prec d \prec e} b^\sigma ( c^\sigma ( d^\sigma e^\sigma ) ),
&\qquad
&
a(((bc)d)e) = a \big( b \sum_{\tau \in S_3} ( c^\tau ( d^\tau e^\tau ) ) \big),
\\
&
a((b(cd))e) = a \big( b \sum_{c \prec d} ( c^\tau ( d^\tau e^\tau ) ) \big),
&\qquad
&
a((bc)(de)) = a \big( b \sum_{d \prec e} ( c^\tau ( d^\tau e^\tau ) ) \big),
\\
&
a(b((cd)e)) = a(b(c(de))) + a(b(c(de))),
&\qquad
&
\text{$a(b(c(de)))$ is already in normal form}.
\end{alignat*}
Figure \ref{zinbielnormalform} presents a recursive algorithm for computing 
the Zinbiel normal form of a nonassociative monomial.
\end{proof}

\begin{figure}[ht]
\fbox{
\begin{tabular}{l}
$\mathtt{ZNF}$: Zinbiel normal form of a nonassociative monomial.
\\
Input: a binary nonassociative monomial $m$.
\\
Output: the Zinbiel normal form of $m$ (sum of right-normed monomials).
\\
If $m$ is indecomposable then: 
\\
\qquad
Set $\mathtt{output} \leftarrow [ m ]$ (the list containing only $m$).
\\
else ($m$ is decomposable):
\\
\qquad
Write $m = wz$.
\\
\qquad 
If $w$ is indecomposable then:
\\
\qquad\qquad
Set $\mathtt{output} \leftarrow [ \; wt \mid t \in \mathtt{ZNF}( z ) \; ]$. 
\\
\qquad\qquad
(Notation: square brackets indicate a list or multiset.)
\\
\qquad
else ($w$ is decomposable):
\\
\qquad\qquad
Write $w = xy$.
\\
\qquad\qquad 
Set $\mathtt{output} \leftarrow [ \; ]$ (the empty list).
\\ 
\qquad\qquad 
For $r \in \mathtt{ZNF}( x )$ do
for $s \in \mathtt{ZNF}( y )$ do
for $t \in \mathtt{ZNF}( z )$ do:
\\ 
\qquad\qquad\qquad 
For $p \in \mathtt{ZNF}( r(st) )$ do: Append $p$ to $\mathtt{output}$.
\\
\qquad\qquad\qquad 
For $q \in \mathtt{ZNF}( r(ts) )$ do: Append $q$ to $\mathtt{output}$.
\\
Return $\mathtt{output}$.
\end{tabular}
}
\vspace{-3mm}
\caption{Recursive algorithm to compute the Zinbiel normal form}
\label{zinbielnormalform}
\end{figure}

\begin{lemma}
\label{expansionlemma}
The expansions of the ternary monomials $[[a,b,c],d,e]$ and $[a,b,[c,d,e]]$ have the following 
Zinbiel normal forms:
\begin{align*}
&
[[a,b,c],d,e] =
\left[
\begin{array}{c}
++++++++++++++++++++++++
\\[-2pt]
------------------------
\\[-2pt]
------++++++--++-+--++-+
\\[-2pt]
------++++++++--+---+++-
\\[-2pt]
------++++++++--+-++---+
\end{array}
\right]
\\
&
[a,b,[c,d,e]] =
\left[
\begin{array}{c}
++---+++++++-------+--++
\\[-2pt]
--+++-------+++++++-++--
\\[-2pt]
------++++++--++-+--++-+
\\[-2pt]
++++++------++--+-++--+-
\\[-2pt]
+-++---+--++++--+---++-+
\end{array}
\right]
\end{align*}
The notation is as follows: each normal form consists of a linear combination of the 120 right-normed
monomials consisting of all permutations of the arguments $a,b,c,d,e$; every coefficient is $\pm 1$.
The $5 \times 24$ sign matrices displayed above contain the sequence of coefficients in which the 
standard row-column order of matrix entries corresponds to the lex order of permutations.
\end{lemma}

\begin{proof}
For each ternary monomial, we first use equation \eqref{ttp} twice to compute the expansion,
obtaining a linear combination of 36 nonassociative monomials with coefficients $\pm 1$, and we then
use the algorithm of Figure \ref{zinbielnormalform} to compute the Zinbiel normal form of each
nonassociative monomial; the result is a linear combination with coefficients $\pm 1$ of all 120
right-normed monomials.
\end{proof}

\begin{lemma}
\label{E5lemma}
With respect to the ordered basis of Lemma \ref{skewbasis5}, the coefficient vectors of the relations 
of arity 5 satisfied by the tortkara triple product may be identified with the nonzero vectors in 
the nullspace of the $120 \times 90$ matrix $E_5$ defined as follows: the $(i,j)$ entry of $E_5$
is the coefficient of the $i$-th basis monomial for $\mathbf{Zinb}(5)$ in lex order 
(Lemma \ref{zinbielbasis5}) 
in the expansion of the $j$-th basis monomial for $\mathbf{SkewTS}(5)$ in lex order 
(Lemma \ref{skewbasis5}),
where the expansions are obtained by applying the appropriate permutation $\sigma \in S_5$ to
the arguments of the two basic expansions (Lemma \ref{expansionlemma}).
\end{lemma}

\begin{proof}
The matrix $E_5$ represents the expansion map $\mathbf{SkewTS}(5) \longrightarrow \mathbf{Zinb}(5)$
with respect to the given ordered bases of the homogeneous components.
\end{proof}

\begin{theorem}
\label{theorem5}
The nullspace of the expansion matrix $E_5$ of Lemma \ref{E5lemma} has dimension 30 and is generated 
as an $S_5$-module by the following relation:
\begin{align*}
&
TT(a,b,c,d,e) =
  [[a,b,c],d,e] 
- [[a,c,b],d,e] 
+ [[b,c,d],e,a] 
+ [[b,d,c],a,e] 
\\[-1pt]
& {}
- [[b,e,d],c,a] 
- [[c,d,b],a,e] 
+ [[c,d,b],e,a] 
+ [[d,e,b],c,a]
+ [a,c,[b,d,e]] 
\\[-1pt]
& {}
- [a,e,[c,d,b]] 
- [c,d,[a,b,e]] 
+ [c,d,[a,e,b]] 
- [c,d,[b,e,a]] 
- [d,e,[b,c,a]].
\end{align*}
\end{theorem}

\begin{proof}
Over a field $\mathbb{F}$, a basis for the nullspace of an $m \times n$ matrix $E$ of rank $r$
may be computed by the standard algorithm: first compute the row canonical form (RCF) of $E$, 
identify the $n{-}r$ free columns (those which do not contain one of the $r$ leading 1s), set the 
free variables equal to the standard basis vectors in $\mathbb{F}^{n-r}$, and solve for the 
leading variables.
However, for a matrix $E$ with integer entries, there is a fraction-free method: since $\mathbb{Z}$
is a Euclidean domain (and hence a PID), we may combine Gaussian elimination and the Euclidean
algorithm for GCDs to compute the Hermite normal form (HNF) of $E$.
More precisely, we compute the HNF, say $H$, of the transpose $E^t$, and keep track of the row
operations to produce an invertible integer matrix $U$ for which $UE^t = H$.
Since the bottom $n{-}r$ rows of $H$ are zero, it follows that the bottom $n{-}r$ rows of $U$
belong to (in fact form a basis of) the right nullspace of $E^t$ (which is the left nullspace
of $E$); we call this $(n{-}r) \times n$ matrix $N$.
Furthermore, we may remain in the category of free $\mathbb{Z}$-modules (instead of vector spaces
over $\mathbb{F}$), and apply the LLL algorithm for lattice basis reduction to $N$ to produce
another matrix $N'$ whose rows generate the same free $\mathbb{Z}$-module as the rows of $N$,
but whose Euclidean lengths are much smaller.
As a measure of the size of the basis, we use the base 10 logarithm of the product of the 
squared lengths of the basis vectors.
For the original basis $N$ this measure is $\approx 40.847$; after applying the LLL algorithm 
with the standard value $3/4$ of the reduction parameter, we obtain a basis with measure 
$39.851$; using the higher reduction parameter $999/1000$, we obtain a basis $N'$ with 
measure $\approx 35.656$.
The square lengths of the rows of the original basis $N$ are as follows 
(multiplicities are written as exponents):
$14^3$, $16^3$, $20^2$, $22$, $24^{13}$, $26$, $28^2$, $30$, $32^2$, $36^2$; for the reduced
basis $N'$ the square lengths are $14^{13}$, $16^{14}$, $18$, $20$, $22$.
The shortest basis vector has not changed, but the total length of the basis has dropped by
more than $10^5$; the shortest vector corresponds to the relation $TT(a,b,c,d,e)$.
We then apply all 120 permutations of $a, b, c, d, e$ to $TT(a,b,c,d,e)$, store
the coefficient vectors in a $120 \times 90$ matrix, compute its RCF over $\mathbb{Q}$, and
verify that this RCF coincides with the RCF of $N$. 
For more about the application of LLL to identical relations for nonassociative structures, see 
\cite{BP2009}.
\end{proof}

\begin{corollary}
\label{corollary5}
Let $[\lambda]$ denote the irreducible representation of $S_5$ corresponding to the partition
$\lambda$ of 5.
The nullspace of the expansion matrix $E_5$ of Lemma \ref{E5lemma} has the following $S_5$-module
structure where the exponents indicate multiplicities:
\[
[221] \oplus [311] \oplus [32]^2 \oplus [41]^2 \oplus [5].
\]
\end{corollary}

\begin{proof}
Using the reduced basis $N'$ from the proof of Theorem \ref{theorem5}, we compute the $30 \times 30$
matrices representing the action of a set of conjugacy class representatives on the nullspace of $E_5$, 
and obtain the character $[30, -6, 2, 0, 0, 0, 0]$.
Comparing this with the character table for $S_5$, we obtain the indicated decomposition.
\end{proof}


\section{Relations of arity 7}

The (somewhat redundant) consequences in arity 7 of the relation $TT(a,b,c,d,e)$ 
with respect to operadic partial composition are these eight relations:
\begin{equation}
\label{partial}
\left\{ \;
\begin{array}{l@{\;\;\;}l@{\;\;\;}l}
TT([a,f,g],b,c,d,e), &
TT(a,[b,f,g],c,d,e), &
TT(a,b,[c,f,g],d,e),
\\[2pt]
TT(a,b,c,[d,f,g],e), &
TT(a,b,c,d,[e,f,g]),
\\[2pt] {}
[TT(a,b,c,d,e),f,g], &
[f,TT(a,b,c,d,e),g], &
[f,g,TT(a,b,c,d,e)].
\end{array}
\right.
\end{equation}

\begin{theorem}
\label{theorem7}
The multilinear relation in arity 7 displayed in Figure \ref{newarity7} has 60 terms and coefficients 
$\pm 1$, $\pm 2$; it is satisfied by the tortkara triple product in the Zinbiel operad, but it is not 
a consequence of the skewsymmetry in arity 3 or the 14-term relation in arity 5 displayed in Theorem 
\ref{theorem5}.

Let $\mathrm{Con}(7) \subset \mathbf{SkewTS}(7)$ be the $S_7$-submodule generated by the consequences 
\eqref{partial} with respect to operadic partial composition of the relation $TT(a,b,c,d,e)$ displayed in 
Theorem \ref{theorem5}, and let $\mathrm{ConNew}(7) \subset \mathbf{SkewTS}(7)$ be the $S_7$-submodule 
generated by those consequences together with the 60-term relation in Figure \ref{newarity7}.
The quotient module $\mathrm{ConNew}(7)/\mathrm{Con}(7)$ has dimension 106 and the following 
multiplicity-free decomposition into the direct sum of irreducible representations:
\begin{equation}
\label{connew7}
\mathrm{ConNew}(7)/\mathrm{Con}(7) \cong [421] \oplus [322] \oplus [3211] \oplus [31111].
\end{equation}
The dimensions of the irreducible summands are respectively 35, 21, 35, 15.

Let $\mathrm{All}(7) \subset \mathbf{SkewTS}(7)$ be the $S_7$-module consisting of all relations 
satisfied by the tortkara triple product in arity 7.
The quotient module $\mathrm{All}(7)/\mathrm{Con}(7)$ has dimension 246; it decomposes as follows:
\begin{equation}
\label{isotypic7}
\mathrm{All}(7)/\mathrm{Con}(7) 
\cong 
[511] \oplus [421]^2 \oplus [4111]^2 \oplus [322] \oplus [3211]^2 \oplus [31111]^2.
\end{equation}
The dimensions of the isotypic components are respectively 15, 70, 40, 21, 70, 30.
\end{theorem}

\begin{figure}[ht]
\[
\boxed{
\begin{array}{l}
   2  [dg[[efb]ca]]  
  -2  [dg[[efc]ba]]  
  +2  [ef[[dgb]ca]]  
  -2  [ef[[dgc]ba]]  
  +2  [dg[ef[abc]]]  
\\[2pt] {}
  -2  [dg[ef[acb]]]  
  +2  [ef[dg[abc]]]  
  -2  [ef[dg[acb]]]  
  -   [[abd]g[efc]]  
  -   [[abe]f[dgc]]  
\\[2pt] {}
  +   [[abf]e[dgc]]  
  +   [[abg]d[efc]]  
  +   [[acd]g[efb]]  
  +   [[ace]f[dgb]]  
  -   [[acf]e[dgb]]  
\\[2pt] {}
  -   [[acg]d[efb]]  
  +   [[adb]g[efc]]  
  -   [[adc]g[efb]]  
  +   [[adg]b[efc]]  
  -   [[adg]c[efb]]  
\\[2pt] {}
  +   [[aeb]f[dgc]]  
  -   [[aec]f[dgb]]  
  +   [[aef]b[dgc]]  
  -   [[aef]c[dgb]]  
  -   [[afb]e[dgc]]  
\\[2pt] {}
  +   [[afc]e[dgb]]  
  -   [[afe]b[dgc]]  
  +   [[afe]c[dgb]]  
  -   [[agb]d[efc]]  
  +   [[agc]d[efb]]  
\\[2pt] {}
  -   [[agd]b[efc]]  
  +   [[agd]c[efb]]  
  -   [[bda]g[efc]]  
  -   [[bdg]a[efc]]  
  -   [[bea]f[dgc]]  
\\[2pt] {}
  -   [[bef]a[dgc]]  
  +   [[bfa]e[dgc]]  
  +   [[bfe]a[dgc]]  
  +   [[bga]d[efc]]  
  +   [[bgd]a[efc]]  
\\[2pt] {}
  +   [[cda]g[efb]]  
  +   [[cdg]a[efb]]  
  +   [[cea]f[dgb]]  
  +   [[cef]a[dgb]]  
  -   [[cfa]e[dgb]]  
\\[2pt] {}
  -   [[cfe]a[dgb]]  
  -   [[cga]d[efb]]  
  -   [[cgd]a[efb]]  
  -   [[dga]b[efc]]  
  +   [[dga]c[efb]]  
\\[2pt] {}
  +   [[dgb]a[efc]]  
  -2  [[dgb]c[efa]]  
  -   [[dgc]a[efb]]  
  +2  [[dgc]b[efa]]  
  -   [[efa]b[dgc]]  
\\[2pt] {}
  +   [[efa]c[dgb]]  
  +   [[efb]a[dgc]]  
  -2  [[efb]c[dga]]  
  -   [[efc]a[dgb]]  
  +2  [[efc]b[dga]]  
\end{array}
}
\]
\vspace{-6mm}
\caption{New relation in arity 7 for the tortkara triple product}
\label{newarity7}
\end{figure}

\begin{proof}
For arity 7, the sizes of the matrices involved are too large to use rational arithmetic, so we use 
arithmetic modulo $p = 101$ to keep the time and space requirements within reasonable bounds.

\emph{Step 1:}
We first construct the consequences \eqref{partial} of the new relation of arity 5; the monomials
in these consequences require straightening (using anticommutativity in the first two arguments)
in order to ensure that they involve only the six standard association types in arity 7:
\begin{equation}
\label{types7}
\left\{ \quad
\begin{array}{l@{\qquad}l@{\qquad}l}
[[[\ast,\ast,\ast],\ast,\ast],\ast,\ast], & 
[[\ast,\ast,[\ast,\ast,\ast]],\ast,\ast], & 
[\ast,\ast,[[\ast,\ast,\ast],\ast,\ast]],   
\\[2pt] {}
[\ast,\ast,[\ast,\ast,[\ast,\ast,\ast]]], & 
[[\ast,\ast,\ast],[\ast,\ast,\ast],\ast], & 
[[\ast,\ast,\ast],\ast,[\ast,\ast,\ast]].   
\end{array}
\right.
\end{equation}
These association types have respectively 1, 2, 2, 3, 3, 2 skew-symmetries:
\begin{equation}
\label{symmetries7}
\left\{ \quad
\begin{array}{r@{\qquad}r}
[[[abc]de]fg] + [[[bac]de]fg] \equiv 0,
\\[1pt] {}
[[ab[cde]]fg] + [[ba[cde]]fg] \equiv 0,
&
[[ab[cde]]fg] + [[ab[dce]]fg] \equiv 0,
\\[1pt] {}
[ab[[cde]fg]] + [ba[[cde]fg]] \equiv 0,
&
[ab[[cde]fg]] + [ab[[dce]fg]] \equiv 0,
\\[1pt] {}
[ab[cd[efg]]] + [ba[cd[efg]]] \equiv 0,
&
[ab[cd[efg]]] + [ab[dc[efg]]] \equiv 0,
\\[1pt] {}
&
[ab[cd[efg]]] + [ab[cd[feg]]] \equiv 0,
\\[1pt] {}
[[abc][def]g] + [[bac][def]g] \equiv 0,
&
[[abc][def]g] + [[abc][edf]g] \equiv 0,
\\[1pt] {}
&
[[abc][def]g] + [[def][abc]g] \equiv 0,
\\[1pt] {}
[[abc]d[efg]] + [[bac]d[efg]] \equiv 0,
&
[[abc]d[efg]] + [[abc]d[feg]] \equiv 0.
\end{array}
\right.
\end{equation}
The total number of multilinear monomials is therefore
\[
\dim \mathbf{SkewTS}(7) 
= 
\left( \tfrac12 + \tfrac14 + \tfrac14 + \tfrac18 + \tfrac18 + \tfrac14 \right) 7!
=
7560.
\]

\emph{Step 2:}
We initialize to zero a block matrix of size $( 7560 + 5040 ) \times 7560$.
For each of the eight consequences \eqref{partial} we apply all permutations of the arguments,
straighten using skewsymmetry, store the resulting coefficient vectors in the lower block
($5040 \times 7560$), and compute the row canonical form (RCF); nonzero entries in the upper
block are retained for subsequent iterations.
The cumulative ranks obtained are 1785, 2730, 3150, 3150, 3150, 4410, 4410, 4794 which shows that
consequences 4, 5, 7 are not required as generators of the $S_7$-module $\mathrm{Con}(7)$ of 
all consequences of $TT(a,b,c,d,e)$ in arity 7 and that $\dim \mathrm{Con}(7) = 4794$.

\emph{Step 3:}
We insert the identity permutation of the arguments $a, \dots, g$ into the association
types \eqref{types7}, expand them using the tortkara triple product \eqref{ttp}, obtaining six
sums each with 216 terms and coefficients $\pm 1$, normalize the terms using the algorithm of
Figure \ref{zinbielnormalform}, and sort the resulting permutations in lex order, obtaining six 
sequences of $\pm$ signs of length 5040 which are the analogues in arity 7 of the two sequences
of length 120 in Lemma \ref{expansionlemma}.
For each association type $t$, we apply the permutations corresponding to the multilinear monomials
of type $t$ to the expansion of type $t$, and store the resulting sequences of $\pm$ signs in
the columns of the expansion matrix of size $5040 \times 7560$.
We compute the RCF and find that the rank is 2520, so the nullity is 5040.
The nullspace of this matrix is the $S_7$-module $\mathrm{All}(7)$ containing all the (multilinear) 
identical relations satisfied by the tortkara triple product in arity 7.
Comparing this with the result of Step 1, we see that the quotient module of new relations,
$\mathrm{New}(7) = \mathrm{All}(7) / \mathrm{Con}(7)$, has dimension $5040 - 4794 = 246$.

\emph{Step 4:}
From the full rank matrix of size $2520 \times 7560$ produced by Step 2, we extract a $5040 \times 7560$
matrix whose row space is the nullspace of the expansion matrix, and compute its RCF.
We then sort the rows of the RCF by increasing number of nonzero entries, obtaining a minimum of 17
and a maximum of 1397.
We then process these sorted rows as in Step 1 to determine which of the nullspace basis vectors 
do not belong to the $S_7$-module generated by the consequences of $TT(a,b,c,d,e)$ and the
previously processed nullspace basis vectors.
A relation with 60 terms increases the rank 
from $\dim \mathrm{Con}(7) = 4794$ to $\dim \mathrm{ConNew}(7) = 4900$ 
(this is the relation displayed in Figure \ref{newarity7});
a relation with 798 terms then increases the rank from 4900 to 4970;
and finally a relation with 985 terms increases the rank from 4900 to $\dim \mathrm{All}(7) = 5040$.
The coefficients modulo $p = 101$ of the relation with 60 terms are $1, 50, 51, 100$; 
multiplying by 2 and using symmetric representatives gives the coeffients $2, -1, 1, -2$.

\emph{Step 5:}
We confirm the results of Steps 1--3 using the representation theory of the symmetric group.
These computational methods have been described in detail elsewhere \cite{BMP2016}, so we omit
most of the details and provide only a brief outline.
The basic idea is to break the computation down into smaller subproblems corresponding to 
the irreducible representations of $S_7$.
Let $\lambda$ be a partition of 7, with corresponding simple $S_7$-module $[\lambda]$ of
dimension $d_\lambda$.
For each $s \in \mathbb{F} S_7$, let $R_\lambda(s)$ be the $d_\lambda \times d_\lambda$ matrix for
$s$ in the natural representation which may be computed according to Clifton's algorithm
\cite[Figure 3]{BMP2016}.
For each $\lambda$, we construct matrices $[\mathrm{Sym}(\lambda)]$, $[\mathrm{Con}(\lambda)]$,
$[\mathrm{New}(\lambda)]$ of sizes $13d_\lambda \times 6d_\lambda$, $8d_\lambda \times 6d_\lambda$,
$d_\lambda \times 6d_\lambda$ respectively, which contain the blocks $R_\lambda(s)$ representing
the terms in each association type for the symmetries \eqref{symmetries7}, the consequences 
\eqref{partial}, and the new 60-term relation (Figure \ref{newarity7}).
We stack these matrices vertically and compute the RCFs; the ranks give the multiplicities of
the irreducible representation $[\lambda]$ in the corresponding $S_7$-modules (Figure \ref{mults7}).
Comparing the last two rows we obtain the decomposition \eqref{connew7}.

\begin{figure}[ht]
\[
\begin{array}{cc@{\;\;}c@{\;\;}c@{\;\;}c@{\;\;}c@{\;\;}c@{\;\;}c@{\;\;}c@{\;\;}c@{\;\;}c@{\;\;}c@{\;\;}c@{\;\;}c@{\;\;}c@{\;\;}c}
& 7 & 61 & 52 & 51^2 & 43 & 421 & 41^3 & 3^21 & 32^2 & 321^2 & 31^4 & 2^31 & 2^21^3 & 21^5 & 1^7 
\\ \midrule
\begin{bmatrix}
\mathrm{Sym}(\lambda)
\end{bmatrix}
& 6 & 35 & 77 & 81 & 71 & 172 & 95 & 95 & 92 & 145 & 57 & 50 & 44 & 14 & 0
\\[1mm]
\begin{bmatrix}
\mathrm{Sym}(\lambda) \\ 
\mathrm{Con}(\lambda) 
\end{bmatrix}
& 6 & 35 & 80 & 84 & 79 & 193 & 108 & 116 & 114 & 188 & 78 & 75 & 74 & 31 & 5 
\\[3mm]
\begin{bmatrix}
\mathrm{Sym}(\lambda) \\ 
\mathrm{Con}(\lambda) \\
\mathrm{New}(\lambda) \\
\end{bmatrix}
& 6 & 35 & 80 & 84 & 79 & 194 & 108 & 116 & 115 & 189 & 79 & 75 & 74 & 31 & 5 
\\ \midrule
\end{array}
\]
\vspace{-7mm}
\caption{Multiplicities I of irreducible representations in arity 7}
\label{mults7}
\end{figure}

\begin{figure}[ht]
\[
\begin{array}{cc@{\;\;}c@{\;\;}c@{\;\;}c@{\;\;}c@{\;\;}c@{\;\;}c@{\;\;}c@{\;\;}c@{\;\;}c@{\;\;}c@{\;\;}c@{\;\;}c@{\;\;}c@{\;\;}c}
& 7 & 61 & 52 & 51^2 & 43 & 421 & 41^3 & 3^21 & 32^2 & 321^2 & 31^4 & 2^31 & 2^21^3 & 21^5 & 1^7 
\\ \midrule
\begin{bmatrix}
\mathrm{Exp}(\lambda)
\end{bmatrix}
& 0 & 1 & 4 & 5 & 5 & 15 & 10 & 10 & 11 & 20 & 10 & 9 & 10 & 5 & 1
\\[1mm]
\begin{bmatrix}
\mathrm{Nul}(\lambda) 
\end{bmatrix}
& 6 & 35 & 80 & 85 & 79 & 195 & 116 & 115 & 114 & 190 & 80 & 75 & 74 & 31 & 5 
\\ \midrule
\end{array}
\]
\vspace{-7mm}
\caption{Multiplicities II of irreducible representations in arity 7}
\label{expnul7}
\end{figure}

\emph{Step 6:}
For each $\lambda$, we construct the matrix $[\mathrm{Exp}(\lambda)]$ 
of size $d_\lambda \times 6d_\lambda$ in which
the $k$-th block $R_\lambda(s)^t$ is the transpose of the representation matrix 
for the normalized Zinbiel expansion of the skew-ternary monomial with the identity permutation of
$a, \dots, g$ in association type $k$.
(For an explanation of the transpose, see \cite[\S5]{BM2013}.)
From this we compute the matrix $[\mathrm{Nul}(\lambda)]$ in RCF whose row space 
is the isotypic component of type $[\lambda]$ in the $S_7$-module $\mathrm{All}(7)$.
Comparing the last row of Figure \ref{expnul7} with the second-last row of Figure \ref{mults7}
gives the decomposition \eqref{isotypic7}.
\end{proof}



\end{document}